\documentclass[times]{iapress}

\usepackage{moreverb}

\usepackage[colorlinks,bookmarksopen,bookmarksnumbered,citecolor=red,urlcolor=red]{hyperref} 
\usepackage{float}
\usepackage{epstopdf}

\def\volumeyear{201x}
\def\volumenumber{x}
\def\volumemonth{Month}
\usepackage{amsmath,amssymb}
\usepackage{graphicx}

\newtheorem{lemma}{Lemma}

\newcommand{\lp}[1]{\left( \begin{array}{#1} }
\newcommand{\rp}{\end{array} \right)}

\begin{document}

\runningheads{A weighted transmuted exponential distribution with environmental applications}{C. Chesneau, H. S. Bakouch and M. N. Khan}

\title{A weighted transmuted exponential distribution with environmental applications}

\author{Christophe Chesneau \affil{1}$^,$\corrauth,
   Hassan S. Bakouch \affil{2}, Muhammad Nauman Khan \affil{3}
 }

\address{
\affilnum{1}Department of Mathematics, LMNO, University of Caen Normandy, France. \\
\affilnum{2}Department of Mathematics, Faculty of Science, Tanta University, Tanta, Egypt.\\
 \affilnum{3}Institute of Numerical Sciences, Kohat University of Science and Technology, Kohat, 26000, Pakistan.\\
}

\corraddr{Christophe Chesneau (Email: christophe.chesneau@unicaen.fr). Department of Mathematics, LMNO, University of Caen Normandy, France.}

\begin{abstract}
In this paper, we introduce a new three-parameter distribution  based on the combination of  re-parametrization of the so-called EGNB2 and  transmuted exponential distributions. This combination aims to modify the transmuted exponential distribution via the incorporation of an additional parameter, mainly adding a high degree of flexibility on the mode and impacting the skewness and kurtosis of the tail. We explore some mathematical  properties of this distribution including the hazard rate function, moments, the moment generating function, the quantile function, various entropy measures and (reversed) residual life functions. A statistical study  investigates estimation of the parameters using the method of maximum likelihood. The distribution along with other existing distributions are fitted to two environmental data sets and its superior performance is assessed by using some goodness-of-fit tests. As a result, some environmental measures associated with these data are obtained such as the return level and mean deviation about this level.
\end{abstract}


\maketitle

\section{Introduction}

The precise analysis of a wide variety of data sets is limited by the use of models based on the classical distributions (normal, exponential, logistic\ldots). For instance, the analysis of environmental data sets collecting from observations of complex natural phenomena needs special treatments to reveal all the underlying  informations. Over the last decades, numerous solutions have been provided by the statisticians, including the elaboration of several methods which aim to increase the flexibility of the former classical distributions. Among these methods, a popular one that aims to construct a generator of distributions by compounding continuous distributions with well-known discrete distributions. This compounding is always  motivated by practical problems as those involving cdf of minimum or maximum of several independent and identically random variables. An exhaustive survey on the construction of such generators, with the presentation of new ones, can be found in \cite{tahir}, and the references therein. Among the long list, let us briefly present the EGNB2 distribution introduced by \cite[Remark 2 (ii)]{tahir}. Using a cumulative distribution function (cdf) $G(x)$, the general form of the associated cdf is given by
\begin{eqnarray}\label{f1}
F_{EGNB2}(x)=\frac{\left[ 1 + \eta   \upsilon G(x)^{\alpha} \right]^{-\frac{1}{\eta}}-1}{(1+\eta   \upsilon)^{-\frac{1}{\eta}}-1}.
\end{eqnarray}
The EGNB2 distribution  can be viewed as an extension of the G-negative binomial families introduced by \cite{louzada} and \cite{percon}. It enjoys remarkable theoretical and practical properties.

In this study, we consider a particular case of this EGNB2 distribution consisting in a re-parametrization for the parameters $\alpha$, $\eta$ and $\upsilon$ appearing in \eqref{f1} as described below.  Let $\gamma>0$, $\eta=-\frac{\gamma}{\gamma+1}$,  $\upsilon=-(\gamma+1)$ and $\alpha=1$. That yields a cdf of the (simple) form:
\begin{eqnarray}\label{f2}
F(x)=\frac{\left[ 1+ \gamma   G(x)\right] ^{\frac{1}{\gamma }+1}-1}{\left( 1+ \gamma \right) ^{\frac{1}{\gamma }+1}-1}.
\end{eqnarray}
Let us now explain the importance of this re-parametrization of \eqref{f1}, with some statistical features. 
One can observe that $F(x)$ as the following integral form: $F(x)=\int_{0}^{G(x)}p(t) \mathrm{d}t$, where $p(t)$ denotes the pdf: $p(t)=\frac{\gamma+1}{\left( 1+ \gamma \right) ^{\frac{1}{\gamma }+1}-1}(1+\gamma  t)^{\frac{1}{\gamma}}$. So it reveals to be a new particular case of the T-X family cdf introduced by \cite{alzaa}. Another remark is that, when $G(x)\rightarrow 0$, we have $F(x)\sim \frac{\gamma+1}{(1+\gamma)^{\frac{1}{\gamma}+1}-1}G(x)$ and when $\gamma\rightarrow 0$, we have $F(x)\sim \frac{\mathrm{e}^{G(x)}-1}{e-1}$. This transformation of cdf corresponds to the one proposed in \cite{kumar1}. All the resulting distributions have demonstrated nice properties in terms of analysis of real life data sets. 
Furthermore, let us observe that the probability density function (pdf) associated to \eqref{f2} is given by
\begin{equation*}
f(x)=\frac{(\gamma+1)\left[ 1+ \gamma  G(x)\right] ^{\frac{1}{\gamma }}g(x)}{\left( 1+ \gamma \right) ^{\frac{1}{\gamma }+1}-1}.
\end{equation*}
Note that we can also express it as a weighted pdf: $f(x)=c w(x) g(x)$, where $w(x)=\left[1 + \gamma   G(x)\right]^{\frac{1}{\gamma}}$ is a weight function and $c=\frac{\gamma +1}{\left( 1+\gamma \right) ^{\frac{1}{\gamma }+1}-1}
$ is a normalizing constant. It thus belongs to the family of weighted distributions. Further details on such family of distributions can be found in  \cite{saphir}.  On the other side, \cite{Aryal and Tsokos:2011} introduced the transmuted exponential distribution defined by the following cdf:   $G(x)=(\theta+1)H(x)-\theta[H(x)]^2$, $\theta \in [-1,1]$, where $H(x)$ denotes the cdf of the exponential distribution. Then, it is proved that the additional parameter $\theta$ can significantly increase the flexibility of the former exponential distribution,  demonstrating a superiority in terms of fit in comparison to the former exponential distribution. We may refer the reader to \cite{ow}, and the references therein. 

 In this paper, we introduce a new three-parameter distribution which combines the features of the distribution characterized by \eqref{f2} and  the transmuted exponential distribution. This combination aims to modify the former transmuted exponential distribution by incorporating the parameter $\gamma$ and takes benefit of the flexibility of the EGNB2 distribution. Its main role is to add a high degree of flexibility on the mode, and the skewness and kurtosis of the tail. We thus obtain a very flexible distribution, which opens new perspectives in terms of the construction of statistical models for data analysis. The theoretical and practical aspects are explored in an exhaustive way. The theoretical ones include expansions of the cdf, pdf, hazard rate function (hrf), quantile function, moments, moment generating function, various entropy measures, residual life functions, conditional moments, mean deviations and reversed residual life function. We investigate the estimation of its parameters via the maximum likelihood method. Two real-life data sets in environmental sciences are analyzed to show its superior performance in terms of fit in comparison to well-known distributions: The gamma distribution, the Marshal-Olkin exponential distribution  \cite{Marshall and Olkin:1997}, the Nadarajah-Haghighi exponential distribution \cite{Nadarajah and Haghighi}, the exponentiated exponential distribution   \cite{gupta2}, the transmuted Weibull distribution  \cite{Aryal and Tsokos:2011}, the transmuted generalized exponential distribution \cite{Khan_King_Hudson}, the transmuted linear exponential distribution \cite{Ref tle} and the Kappa distribution  \cite{Mielke:1973}. The best performance of the proposed distribution recommends it as a hydrologic probability model, such as the most known distributions: Kappa and gamma distributions. This motivates to estimate important hydrologic parameters of those data sets by making use of the distribution.  

The rest of this article is organized as follows. In Section \ref{new}, we present our main distribution. Some of its mathematical properties are studied in Section \ref{struc}. Residual life functions are determined in Section \ref{resi}. Estimations of the parameters are investigated in Section \ref{estim}. Applications to two real-life data sets are provided in Section \ref{appli}. Concluding remarks are addressed in Section \ref{conclusion}.

\section{A new weighted transmuted exponential distribution}\label{new}

In this section, we precise what is the considered cdf $G(x)$ given by \eqref{f2}. \cite{shaw} and \cite{Aryal and Tsokos:2011} introduced the quadratic rank transmutation map (QRTM) to propose a new distribution based on the Weibull/exponential one with great flexibility and nice fit for real-life data. In the current studies, it remains a serious competitor in terms of precision in modelling (see \cite{ow}). 
%
For these reasons, we use it in our study. We consider the cdf: $$G(x)=(\theta+1)H(x)-\theta \left[H(x)\right]^2, \quad  \theta \in [-1,1],$$ where $H(x)$ is considered to be the cdf of the exponential distribution of parameter $\lambda$:
$$G(x)= (\theta+1)(1-\mathrm{e}^{-\lambda  x})-  \theta(1-\mathrm{e}^{-\lambda  x})^2, \quad x, \lambda>0.$$
Set the above expression into \eqref{f2}, we introduce a new cdf defined by
\begin{align*}
F(x)&=\frac{\left[ 1+ \gamma  (\theta+1) (1-\mathrm{e}^{-\lambda   x})-\gamma  \theta (1-\mathrm{e}^{-\lambda  x})^2\right] ^{\frac{1}{\gamma }+1}-1%
}{\left( 1+ \gamma \right) ^{\frac{1}{\gamma }+1}-1}\\
& = \frac{\left[ 1+ \gamma -  \gamma  \mathrm{e}^{-\lambda  x}(1 - \theta + \theta   \mathrm{e}^{-\lambda  x})\right] ^{\frac{1}{\gamma }+1}-1%
}{\left( 1+ \gamma \right) ^{\frac{1}{\gamma }+1}-1}, \quad x>0,  \lambda,   \gamma >0.
\end{align*}
Another useful expression is the following one:
\begin{eqnarray}\label{cdf0}
F(x)=\frac{\left( 1+ \gamma \right) ^{\frac{1}{\gamma }+1}\left[ 1 -  \frac{\gamma}{1+\gamma}  \mathrm{e}^{-\lambda  x}(1 - \theta + \theta \mathrm{e}^{-\lambda  x})\right]^{\frac{1}{\gamma} + 1 } - 1 }{\left( 1 + \gamma \right) ^{\frac{1}{\gamma} + 1 } - 1 }.
\end{eqnarray}
We will refer to the distribution given by \eqref{cdf0} as the new weighted transmuted exponential and denote it by NWTE($\lambda$, $\gamma$, $\theta$) with the considered parameters.

The corresponding pdf is given by
\begin{eqnarray}\label{pdf0}
f(x)=\frac{\lambda  \left( 1+\gamma\right)^{\frac{1}{\gamma }+1} }{\left( 1 + \gamma \right)
^{\frac{1}{\gamma }+1}-1 }\left[  1 -  \frac{\gamma}{1+\gamma} \mathrm{e}^{-\lambda   x}\left(1 - \theta + \theta   \mathrm{e}^{-\lambda  x}\right)\right] ^{\frac{1}{\gamma }}\mathrm{e}^{-\lambda  x} \left( 1 - \theta + 2  \theta  \mathrm{e}^{-\lambda  x}\right).
\end{eqnarray}

The associated hrf is given by
\begin{eqnarray}\label{hrf0}
h(x)= \frac{\lambda \left[ 1 -  \frac{\gamma}{1+\gamma} \mathrm{e}^{-\lambda  x}(1-\theta+\theta \mathrm{e}^{-\lambda  x})\right] ^{
\frac{1}{\gamma }}\mathrm{e}^{-\lambda  x} \left( 1-\theta+2 \theta  \mathrm{e}^{-\lambda  x}\right)}{1-\left[ 1-  \frac{\gamma}{1+\gamma} \mathrm{e}^{-\lambda  x}\left(1-\theta+\theta  \mathrm{e}^{-\lambda  x}\right)\right] ^{\frac{1}{\gamma}+1}}.
\end{eqnarray}

Let us now discuss the possible shapes of pdf \eqref{pdf0} and hrf \eqref{hrf0} as follows.
$$\lim_{x\rightarrow 0}f(x)=\frac{\lambda (1+\theta)  \left( 1+\gamma\right) }{\left( 1+ \gamma \right)^{\frac{1}{\gamma }+1}-1 }, \quad \lim_{x\rightarrow +\infty}f(x)=0.$$
On the other side, we have
$$\lim_{x\rightarrow 0}h(x)=\frac{\lambda (1+\theta)  \left( 1+\gamma\right) }{\left( 1+ \gamma \right)^{\frac{1}{\gamma }+1}-1 }, \quad \lim_{x\rightarrow +\infty}h(x)=\lambda.$$
In order to visualize the wide variety of shapes, some plots of the pdf \eqref{pdf0} and hrf \eqref{hrf0} are given in Figures \ref{pdf plots} and \ref{hrf plots}. We see that $\gamma$ has a great impact on the mode of the NWTE distribution. Moreover, the hrf also exhibits sudden spikes at the end of upside-down bathtub shapes, which manages the model to analyze a non-stationary real-life data.

\begin{figure}[h]
\centering
\includegraphics[width=5.5cm,height=5cm]{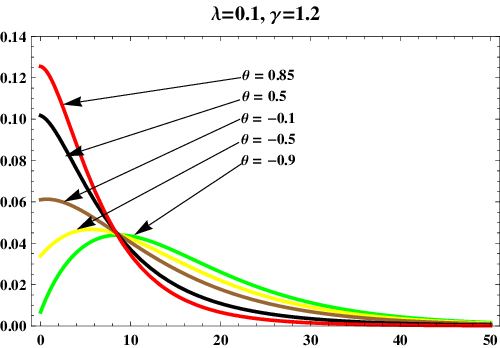}
\hspace{0.15cm}
\includegraphics[width=5.5cm,height=5cm]{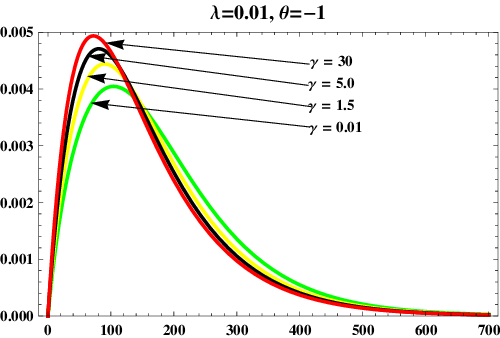}
\caption{Plots of the NWTE pdf.}\label{pdf plots}
\end{figure}

\begin{figure}[h]
\centering
\includegraphics[width=7cm,height=5cm]{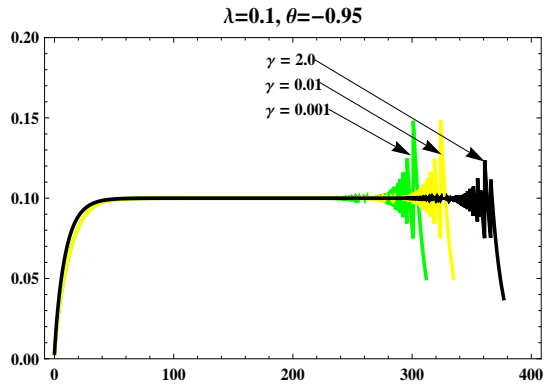}
\caption{Plots of the NWTE hrf.}\label{hrf plots}
\end{figure}

\section{Structural properties of the NWTE distribution}\label{struc}

\subsection{Expansion for the associated functions}

{\it {\textbf{Expansion for the cdf function.}}} First of all, set $h(u)=u(1-\theta+\theta u)$, $u\in (0,1)$, $\theta \in [-1,1]$. Note that we have $h'(u)=1-\theta+2\theta u\ge \min(1-\theta,1+\theta)\ge 0$, so $h$ is increasing.  Since $h(0)=0$ and $h(1)=1$, we have $0< h(u)< 1$ for all $u\in (0,1)$.
Since $0\le \frac{\gamma}{1+\gamma}<1$ and $0<\mathrm{e}^{-\lambda  x}(1-\theta+\theta  \mathrm{e}^{-\lambda  x})= h(\mathrm{e}^{-\lambda  x})<1$, the generalized binomial expansion, we have
\begin{eqnarray}\label{eq1}
\lefteqn{\left[ 1-  \frac{\gamma}{1+\gamma}  \mathrm{e}^{-\lambda  x}(1-\theta+\theta  \mathrm{e}^{-\lambda  x}) \right] ^{\frac{1}{\gamma }+1}=\sum_{i=0}^{+\infty} \binom{1/\gamma+1}{i} \left(-\frac{\gamma}{1+\gamma}\right)^i \mathrm{e}^{-\lambda  i  x}(1-\theta+\theta  \mathrm{e}^{-\lambda  x})^i} & & \nonumber \\
& =
&  \sum_{i=0}^{+\infty}\sum_{k=0}^i\binom{1/\gamma+1}{i}\binom{i}{k} \left(-\frac{\gamma}{1+\gamma}\right)^i \mathrm{e}^{-\lambda i x}(1-\theta)^{i-k}\theta^k \mathrm{e}^{-\lambda k x}=  \sum_{i=0}^{+\infty}\sum_{k=0}^i H_{i,k} \mathrm{e}^{-\lambda  x  (i+k)},
\end{eqnarray}
where
$$H_{i,k}=\binom{1/\gamma+1}{i}\binom{i}{k} \left(\frac{\gamma \theta-\gamma}{1+\gamma}\right)^i  \left( \frac{\theta}{1-\theta}\right)^k.$$
Therefore we can expand the cdf function as
\begin{eqnarray}\label{cdf1}
F(x)=\frac{1
}{\left( 1+ \gamma \right) ^{\frac{1}{\gamma }+1}-1} \left[ (1+\gamma)^{\frac{1}{\gamma }+1} \sum_{i=0}^{+\infty}\sum_{k=0}^i H_{i,k} \mathrm{e}^{-\lambda  x   (i+k)}-1\right].
\end{eqnarray}

{\it  {\textbf{Expansion for the pdf function.}}} Similar mathematical arguments used for \eqref{eq1} give
\begin{eqnarray*}
\left[ 1-  \frac{\gamma}{1+\gamma} \mathrm{e}^{-\lambda  x}(1-\theta+\theta \mathrm{e}^{-\lambda  x}) \right] ^{\frac{1}{\gamma }}=  \sum_{i=0}^{+\infty}\sum_{k=0}^i A_{i,k}   \mathrm{e}^{-\lambda  x  (i+k)},
\end{eqnarray*}
where
\begin{eqnarray*}
A_{i,k}=\binom{1/\gamma}{i}\binom{i}{k} \left(\frac{\gamma  \theta-\gamma}{1+\gamma}\right)^i  \left( \frac{\theta}{1-\theta}\right)^k.
\end{eqnarray*}
Therefore
\begin{eqnarray}\label{e1}
f(x)&=&\frac{\lambda  \left( 1+\gamma\right)^{\frac{1}{\gamma }+1} }{\left( 1+ \gamma \right)
^{\frac{1}{\gamma }+1}-1 }  \sum_{i=0}^{+\infty}\sum_{k=0}^i A_{i,k}   \mathrm{e}^{-\lambda  x  (i+k)}\mathrm{e}^{-\lambda  x} \left( 1-\theta+2 \theta  \mathrm{e}^{-\lambda  x}\right) \nonumber\\
& =
&  \sum_{i=0}^{+\infty}\sum_{k=0}^i B_{i,k}  \left[ (1-\theta) \mathrm{e}^{-\lambda  x (i+k+1)}+2 \theta  \mathrm{e}^{-\lambda  x  (i+k+2)}\right],
\end{eqnarray}
where
$$B_{i,k}=\frac{\lambda  \left( 1+\gamma\right)^{\frac{1}{\gamma }+1} }{\left( 1+ \gamma \right)
^{\frac{1}{\gamma }+1}-1 }A_{i,k}.$$

{\it  {\textbf{On the survival function.}}} Note that
\begin{eqnarray}\label{surv0}
S(x)=1-F(x)=\frac{\left( 1+ \gamma \right) ^{\frac{1}{\gamma }+1}\left[1-\left\{ 1-  \frac{\gamma}{1+\gamma} \mathrm{e}^{-\lambda  x}(1-\theta+\theta \mathrm{e}^{-\lambda  x})\right\} ^{\frac{1}{\gamma}+1}\right]}{\left( 1+ \gamma \right) ^{\frac{1}{\gamma }+1}-1}.
\end{eqnarray}
Using \eqref{cdf1}, we have the following expansion
\begin{eqnarray}\label{sur1}
S(x)=\frac{(1+\gamma)^{\frac{1}{\gamma }+1}
}{\left( 1+ \gamma \right) ^{\frac{1}{\gamma }+1}-1} \left[1- \sum_{i=0}^{+\infty}\sum_{k=0}^i H_{i,k} \mathrm{e}^{-\lambda  x  (i+k)}\right].
\end{eqnarray}

{\it  {\textbf{Expansion for the hrf function.}}} Using \eqref{hrf0}, \eqref{e1} and \eqref{sur1}, an expansion of the hrf function is given by
\begin{eqnarray}\label{hrf1}
h(x)=\frac{\lambda\sum_{i=0}^{+\infty}\sum_{k=0}^i A_{i,k}  \left[ (1-\theta)  \mathrm{e}^{-\lambda  x  (i+k+1)}+2 \theta  \mathrm{e}^{-\lambda x (i+k+2)}\right]}{1- \sum_{i=0}^{+\infty}\sum_{k=0}^i   H_{i,k}\mathrm{e}^{-\lambda x (i+k)}}.
\end{eqnarray}
Another expansion comes from the geometric series decomposition:
\begin{eqnarray*}
\frac{1}{1-\left[ 1-  \frac{\gamma}{1+\gamma} \mathrm{e}^{-\lambda  x}(1-\theta+\theta \mathrm{e}^{-\lambda  x})\right] ^{\frac{1}{\gamma}+1}}=\sum_{m=0}^{+\infty}\left[ 1-  \frac{\gamma}{1+\gamma}  \mathrm{e}^{-\lambda  x}(1-\theta+\theta  \mathrm{e}^{-\lambda  x})\right] ^{m\left(\frac{1}{\gamma}+1\right)}.
\end{eqnarray*}
By \eqref{hrf0} and similar mathematical arguments used for \eqref{eq1} give:
\begin{align*}
h(x)&=\lambda \sum_{m=0}^{+\infty}\left[ 1-  \frac{\gamma}{1+\gamma}  \mathrm{e}^{-\lambda  x}(1-\theta+\theta  \mathrm{e}^{-\lambda  x})\right] ^{m\left(\frac{1}{\gamma}+1\right)+\frac{1}{\gamma}} \left[(1-\theta) \mathrm{e}^{-\lambda  x}+2 \theta  \mathrm{e}^{-2\lambda x}\right]\\
& = \sum_{m=0}^{+\infty}\sum_{i=0}^{+\infty}\sum_{k=0}^i G_{i,k,m}  \left[(1-\theta)  \mathrm{e}^{-\lambda  x  (i+k+1)}+2 \theta  \mathrm{e}^{-\lambda  x  (i+k+2)}\right],
\end{align*}
where
$$G_{i,k,m}=\lambda \binom{m\left(\frac{1}{\gamma}+1\right)+\frac{1}{\gamma}}{i}\binom{i}{k} \left(\frac{\gamma \theta-\gamma}{1+\gamma}\right)^i  \left( \frac{\theta}{1-\theta}\right)^k.$$

\subsection{Quantile function}
The quantile functions are in widespread use in general statistics to obtain mathematical properties of a distribution and often find re\-pre\-sentations in terms of lookup tables for key percentiles. For generating data from the NWTE model, let $u\sim \it{U}(0,1)$. Then, by inverting the cdf \eqref{cdf0} and after some algebra, we get the quantile function
\begin{equation} \label{quantile}
Q(u)=\frac{1}{\lambda}\left\lbrack -\log\left\lbrace 1-\left(\frac{1+\theta-\sqrt{(1+\theta)^2-4\frac{\theta}{\gamma}\left\lbrack\left\{ u\left( \left[ 1+\gamma \right]^{\frac{1}{\gamma}+1}-1 \right) +1\right\}^{\frac{\gamma}{\gamma+1}}-1\right\rbrack }}{2  \theta}\right)\right\rbrace\right\rbrack.
\end{equation}

The analysis of the variability of the skewness and kurtosis of X can be investigated based on quantile measures. The Bowley skewness is given by
\begin{align*}
S=\frac{Q(3/4)-2 Q(1/2)+Q(1/4)}{Q(3/4)-Q(1/4)}
\end{align*}
and the Moors' kurtosis by
\begin{align*}
K = \frac{\{Q(7/8)- Q(5/8)\}+ \{Q(3/8)-Q(1/8)\}}{Q(6/8)-Q(2/8)},
\end{align*}
where $Q(u)$ is given by \eqref{quantile}.

These measures are less sensitive to outliers and they exist even for distributions without moments. Figure \ref{skewness and kurtosis} displays plots of S and K as functions of $\theta$ and $\gamma$, which show their variability in terms of the shape parameters.

\begin{figure}[h]
\centering
\includegraphics[width=5.5cm,height=5cm]{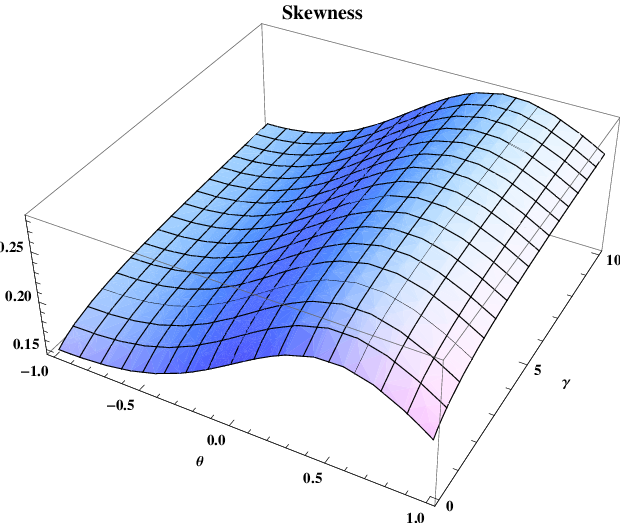}
\hspace{0.15cm}
\includegraphics[width=5.5cm,height=5cm]{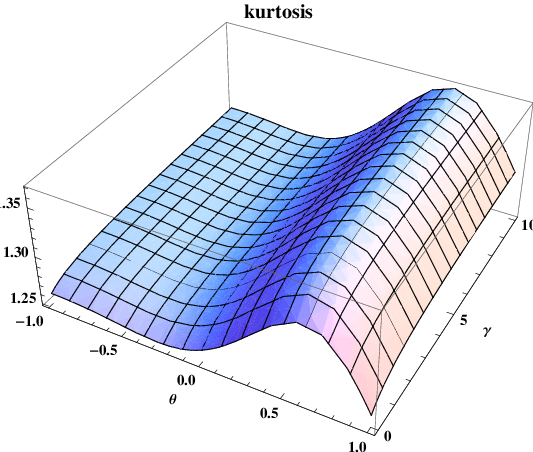}
\caption{Plots of the skewness and kurtosis of the NWTE distribution for $\lambda=0.5$.}\label{skewness and kurtosis}
\end{figure}

\subsection{Moments and moment generating function}

{\it {\textbf{Moments.}}} Using equation \eqref{e1} and the gamma function $\gamma(\nu)=\int_{0}^{+\infty}x^{\nu-1}\mathrm{e}^{-x}\mathrm{d}x$, the $r$-th moments about the origin is given by
\begin{align}\label{mom}
\mu_r&=E(X^r)=\int_{-\infty}^{+\infty}x^rf(x)\mathrm{d}x\nonumber \\
& =\sum_{i=0}^{+\infty}\sum_{k=0}^i B_{i,k}   \left[ (1-\theta)  \int_{0}^{+\infty}x^r  \mathrm{e}^{-\lambda  x  (i+k+1)}\mathrm{d}x + 2 \theta   \int_{0}^{+\infty}x^r \mathrm{e}^{-\lambda  x  (i+k+2)}\mathrm{d}x\right]\nonumber \\
& = \frac{\gamma(r+1)}{\lambda^{r+1}}  \sum_{i=0}^{+\infty}\sum_{k=0}^i B_{i,k}   \left[ \frac{1-\theta}{(i+k+1)^{r+1}} + \frac{2 \theta}{(i+k+2)^{r+1}} \right].
\end{align}

{\it {\textbf{The moment generating function.}}} Similarly the moment generating function associated to the NWTE distribution is given by, for $t\le \lambda$,
\begin{align}\label{mom1}
M_X(t)&=E(\mathrm{e}^{tx})=\int_{-\infty}^{+\infty}\mathrm{e}^{t x} f(x)\mathrm{d}x\nonumber\\
& = \sum_{i=0}^{+\infty}\sum_{k=0}^i B_{i,k}   \left[ (1-\theta)   \int_{0}^{+\infty}\mathrm{e}^{t x}\mathrm{e}^{-\lambda  x  (i+k+1)}\mathrm{d}x+2 \theta  \int_{0}^{+\infty}\mathrm{e}^{t x}\mathrm{e}^{-\lambda  x  (i+k+2)}\mathrm{d}x\right]\nonumber\\
& = \sum_{i=0}^{+\infty}\sum_{k=0}^i B_{i,k}   \left[ \frac{1-\theta}{\lambda(i+k+1)-t} + \frac{2 \theta}{\lambda(i+k+2)-t} \right].
\end{align}

\subsection{Entropies}

An entropy can be considered as a measure of uncertainty of probability distribution of a random variable. Therefore, we obtain three entropies for the NWTE distribution with investigating a numerical study among them.

{\it {\textbf{Entropy 1.}}} Let us consider the Shannon entropy \cite{shannon}: $H(f)=-E\left[\log[f(X)] \right]=-\int_{-\infty}^{+\infty}f(x)\log[f(x)]\mathrm{d}x$. One can observe that
\begin{align}\label{entropy}
H(f) &= -  \log\left[\frac{\lambda  \left( 1+\gamma\right)^{\frac{1}{\gamma }+1} }{\left( 1+ \gamma \right)
^{\frac{1}{\gamma }+1}-1} \right]-\frac{1}{\gamma}\int_{0}^{+\infty}f(x) \log\left[ 1 -  \frac{\gamma}{1+\gamma} \mathrm{e}^{-\lambda  x}(1-\theta+\theta  \mathrm{e}^{-\lambda  x})\right]\mathrm{d}x\nonumber\\
& + \lambda  E(X)-\int_{0}^{+\infty}f(x) \log\left[ 1-\theta+2 \theta  \mathrm{e}^{-\lambda  x}\right]\mathrm{d}x.
\end{align}
Let us now expand the two integrals by using the logarithmic expansion: $\log(1-u)=-\sum_{m=1}^{+\infty}\frac{u^m}{m}$, $|u|<1$. Since $|\frac{\gamma}{1+\gamma} \mathrm{e}^{-\lambda  x}(1-\theta+\theta \mathrm{e}^{-\lambda  x})|<1$, we have
\begin{eqnarray*}
\lefteqn{\int_{0}^{+\infty}f(x) \log\left[ 1 -  \frac{\gamma}{1+\gamma} \mathrm{e}^{-\lambda  x}(1-\theta+\theta \mathrm{e}^{-\lambda  x})\right]\mathrm{d}x}& & \\
& =
&  -\sum_{m=1}^{+\infty}\frac{1}{m}\left(\frac{\gamma}{1+\gamma}\right)^m \int_{0}^{+\infty} f(x)  \mathrm{e}^{-\lambda m x}(1-\theta+\theta  \mathrm{e}^{-\lambda  x})^mdx\\
& =
& -\sum_{m=1}^{+\infty}\sum_{\ell=0}^m\frac{1}{m}\binom{m}{\ell}(1-\theta)^{m-\ell}\left( \frac{\gamma}{1+\gamma}\right)^m\theta^{\ell}\int_{0}^{+\infty}f(x) \mathrm{e}^{-\lambda (\ell+m) x}\mathrm{d}x= \sum_{m=1}^{+\infty}\sum_{\ell=0}^mR_{m,\ell},
\end{eqnarray*}
where
$$R_{m,\ell}=-\frac{1}{m}\binom{m}{\ell}(1-\theta)^{m-\ell}\left( \frac{\gamma}{1+\gamma}\right)^m\theta^{\ell}M_X[-\lambda (\ell+m)],$$
$M_X(t)$ denotes the moment generating function defined by \eqref{mom1}.

For the second integral in \eqref{entropy}, since  $|\theta(1-2\mathrm{e}^{-\lambda  x})|<1$,  we have

\begin{eqnarray*}
\lefteqn{\int_{0}^{+\infty}f(x) \log\left[ 1-\theta+2\theta \mathrm{e}^{-\lambda  x}\right]dx=  -\sum_{m=1}^{+\infty}\frac{\theta^m}{m}\int_{0}^{+\infty}f(x)(1-2 \mathrm{e}^{-\lambda  x})^mdx} & & \\
& =
& -\sum_{m=1}^{+\infty}\sum_{\ell=0}^m\frac{1}{m}\binom{m}{\ell}\theta^m (-2)^{\ell}\int_{0}^{+\infty}f(x) \mathrm{e}^{-\lambda  \ell  x}\mathrm{d}x= \sum_{m=1}^{+\infty}\sum_{\ell=0}^m U_{m,\ell},
\end{eqnarray*}
where
$$U_{m,\ell}=-\frac{1}{m}\binom{m}{\ell}\theta^m(-2)^{\ell}M_X(-\lambda  \ell).$$

{\it {\textbf{Entropy 2.}}} Let us now focus our attention on the  R\'enyi entropy \cite{renyi}: $J_{R}(\beta)=\frac{1}{1-\beta}\log \left( \int_{-\infty}^{+\infty}[f(x)]^{\beta}\mathrm{d}x\right)$, with $\beta\not = 1$ and $\beta>0$. Similar mathematical arguments used for \eqref{eq1} give :
\begin{eqnarray*}
\left[ 1-  \frac{\gamma}{1+\gamma} \mathrm{e}^{-\lambda  x}(1-\theta+\theta \mathrm{e}^{-\lambda  x}) \right] ^{\frac{\beta}{\gamma }}=  \sum_{i=0}^{+\infty}\sum_{k=0}^i C_{i,k}  \mathrm{e}^{-\lambda  x  (i+k)},
\end{eqnarray*}
where
\begin{eqnarray*}
C_{i,k}=\binom{\beta/\gamma}{i}\binom{i}{k} \left(\frac{\gamma  \theta-\gamma }{1+\gamma}\right)^i  \left( \frac{\theta}{1-\theta}\right)^k.
\end{eqnarray*}
On the other side, observing that $|\theta(1-2\mathrm{e}^{-\lambda  x})|<1$, similar mathematical arguments used for \eqref{eq1} give :
\begin{eqnarray*}
\left( 1-\theta+2 \theta  \mathrm{e}^{-\lambda  x} \right) ^{\beta}= \left[ 1-\theta(1-2 \mathrm{e}^{-\lambda  x}) \right] ^{\beta}=  \sum_{j=0}^{+\infty}\sum_{\ell=0}^j D_{j,\ell}   \mathrm{e}^{-\lambda  \ell  x},
\end{eqnarray*}
where
\begin{eqnarray*}
D_{j,\ell}=\binom{\beta}{j}\binom{j}{\ell} (-\theta)^j (-2)^{\ell}.
\end{eqnarray*}
Hence $[f(x)]^{\beta}$ can be expanded as
\begin{eqnarray*}
[f(x)]^{\beta}=  \sum_{i=0}^{+\infty}\sum_{k=0}^i  \sum_{j=0}^{+\infty}\sum_{\ell=0}^j F_{i,k,j,\ell} \mathrm{e}^{-\lambda  x  (i+k+\ell+\beta)},
\end{eqnarray*}
where
$$F_{i,k,j,\ell}=\left[\frac{\lambda  \left( 1+\gamma\right) }{\left( 1+ \gamma \right)
^{\frac{1}{\gamma }+1}-1 }\right]^{\beta} C_{i,k} D_{j,\ell}.$$
Hence
\begin{eqnarray*}
\int_{-\infty}^{+\infty}[f(x)]^{\beta}\mathrm{d}x= \int_{0}^{+\infty} \sum_{i=0}^{+\infty}\sum_{k=0}^i  \sum_{j=0}^{+\infty}\sum_{\ell=0}^j F_{i,k,j,\ell} \mathrm{e}^{-\lambda  x  (i+k+\ell+\beta)}\mathrm{d}x= \frac{1}{\lambda} \sum_{i=0}^{+\infty}\sum_{k=0}^i  \sum_{j=0}^{+\infty}\sum_{\ell=0}^j F_{i,k,j,\ell} \frac{1}{i+k+\ell+\beta}.
\end{eqnarray*}
Therefore
\begin{eqnarray*}
J_{R}(\beta)=\frac{1}{1-\beta }\log \left( \int_{-\infty}^{+\infty}[f(x)]^{\beta}\mathrm{d}x\right)=\frac{1}{1-\beta} \left[- \log(\lambda) + \log\left( \sum_{i=0}^{+\infty}\sum_{k=0}^i  \sum_{j=0}^{+\infty}\sum_{\ell=0}^j F_{i,k,j,\ell}\frac{1}{i+k+\ell+\beta}\right)\right].
\end{eqnarray*}

{\it{\textbf{ Entropy 3.}}} We now focus our attention on the entropy introduced by \cite{Mathai and Haubold:2008}: $J_{MH}(\delta )=\frac{1}{\delta-1}\left(\int_{-\infty}^{+\infty}\left[ f(x)\right]^{2-\delta }\mathrm{d}x-1\right)$, with $\delta\not = 1$ and $\delta>0$. Proceeding as for $J_{R}(\beta)$ with $2-\delta$ instead of $\beta$, we obtain
\begin{eqnarray*}
[f(x)]^{2-\delta}=  \sum_{i=0}^{+\infty}\sum_{k=0}^i  \sum_{j=0}^{+\infty}\sum_{\ell=0}^j G_{i,k,j,\ell} \mathrm{e}^{-\lambda  x  (i+k+\ell+2-\delta)},
\end{eqnarray*}
where
$$G_{i,k,j,\ell}=\left[\frac{\lambda  \left( 1+\gamma\right) }{\left( 1+ \gamma \right)
^{\frac{1}{\gamma }+1}-1 }\right]^{2-\delta}\binom{(2-\delta)/\gamma}{i}\binom{i}{k} \left(\frac{\gamma  \theta-\gamma }{1+\gamma}\right)^i  \left( \frac{\theta}{1-\theta}\right)^k\binom{2-\delta}{j}\binom{j}{\ell} (-\theta)^j (-2)^{\ell}.$$
Hence
\begin{eqnarray*}
J_{MH}(\delta )&=&\frac{1}{\delta-1}\left(\int_{-\infty}^{+\infty}\left[ f(x)\right]^{2-\delta }\mathrm{d}x-1\right)\\
& =
& \frac{1}{\delta-1}\left(\frac{1}{\lambda}\sum_{i=0}^{+\infty}\sum_{k=0}^i  \sum_{j=0}^{+\infty}\sum_{\ell=0}^j G_{i,k,j,\ell}\frac{1}{i+k+\ell+2-\delta}-1\right).
\end{eqnarray*}

Some numerical values for the three entropies are given in Table \ref{Entropies}. It can be observed that these entropies decrease with increasing the parameter values. Moreover, one can see that $J_{MH}(\delta)$ has the smallest values comparing with the other entropies considered here.

\begin{table}
\centering
\caption{Entropy for several arbitrary parameter values with $\lambda=1$.} \label{Entropies}
\scriptsize{\begin{tabular}{llllllllll}
\noalign{\smallskip}\hline\noalign{\smallskip}
$ \gamma \downarrow$ \ \  $\theta=0.5$ & $H(f)$& $J_R(0.5)$ & $J_{MH}(0.5)$\\
\hline
0.1&   0.94207  &  1.32902  &  0.63437 \\
0.4&   0.91244  &  1.30828  &  0.61079 \\
0.8&   0.88415  &  1.28871  &  0.58774 \\
1.2&   0.86349  &  1.27456  &  0.57056 \\
1.5&   0.85122  &  1.26622  &  0.56021 \\
1.8&   0.84092  &  1.25926  &  0.55144 \\
2.0&   0.83492  &  1.25522  &  0.54629 \\ \hline
$ \theta \downarrow$ \ \ $\gamma=0.7$& \\
\hline
-0.9&  1.40499  &  1.66797  &  0.94970 \\
-0.5&  1.33563  &  1.61265  &  0.91027 \\
-0.2&  1.24327  &  1.54839  &  0.84921 \\
0.1&   1.11731  &  1.46168  &  0.75998 \\
0.4&   0.95418  &  1.34297  &  0.64014 \\
0.6&   0.82110  &  1.23549  &  0.54159 \\
0.8&   0.66351  &  1.08602  &  0.42654 \\ \hline
 \noalign{\smallskip}
\end{tabular}}
\end{table}

\subsection{Conditional moments and mean deviations}
Here, we introduce an important lemma which will be used in the next sections.
\begin{lemma}\label{lem1} Let $J_r(t)=\int_{0}^{t}x^rf(x)\mathrm{d}x$ and $\gamma(t,\nu)=\int_{0}^{t}x^{\nu-1}\mathrm{e}^{-x}\mathrm{d}x$ be the lower incomplete gamma function. Then we have
\begin{eqnarray}\label{Jr}
J_r(t)= \frac{1}{\lambda^{r+1}} \sum_{i=0}^{+\infty}\sum_{k=0}^i  B_{i,k}  \left[ (1-\theta)\frac{\gamma\{\lambda(i+k+1)t,r+1\}}{(i+k+1)^{r+1}} + 2 \theta \frac{\gamma\{\lambda(i+k+2)t,r+1\}}{(i+k+2)^{r+1}} \right].
\end{eqnarray}
\end{lemma}
\begin{proof} Using the equation \eqref{e1}, we have
\begin{align*}
J_r(t)&=\int_{0}^{t}x^rf(x) \mathrm{d}x= \sum_{i=0}^{+\infty}\sum_{k=0}^i B_{i,k}  \left[ (1-\theta)  \int_{0}^{t}x^r\mathrm{e}^{-\lambda x (i+k+1)}\mathrm{d}x+2\theta  \int_{0}^{t}x^re^{-\lambda x (i+k+2)}\mathrm{d}x\right]\\
& =\frac{1}{\lambda^{r+1}}  \sum_{i=0}^{+\infty}\sum_{k=0}^i   B_{i,k}  \left[ (1-\theta)\frac{\gamma\{\lambda(i+k+1)t,r+1\}}{(i+k+1)^{r+1}} + 2 \theta\frac{\gamma\{\lambda(i+k+2)t,r+1\}}{(i+k+2)^{r+1}} \right].
\end{align*}
\end{proof}
%
The $r$-th conditional moments of the NWTE distribution is given by
\begin{eqnarray}\label{cm}
E(X^r \mid X>t)=\frac{1}{1-F(t)}\int_{t}^{+\infty}x^r f(x) \mathrm{d}x=\frac{1}{S(t)}\left[E(X^r)-J_r(t)\right].
\end{eqnarray}
It can be expressed using \eqref{hrf0}, \eqref{mom} and Lemma \ref{lem1}. The same remark holds for the $r$-th reversed moments of the NWTE distribution  given by
\begin{eqnarray*}
E(X^r \mid X\le t)=\frac{1}{F(t)}\int_{0}^{t}x^rf(x)\mathrm{d}x=\frac{1}{F(t)}J_r(t).
\end{eqnarray*}
The mean deviations of $X$ about the mean $\mu=E(X)$ can be expressed as $\delta=2\mu F(\mu)-2J_1(\mu)$ and the mean deviations of $X$ about the median $M$ has the form $\eta=\mu-2 J_1(M)$.

\section{(Reversed) Residual life functions}\label{resi}

\subsection{Residual lifetime function}

The residual life is described by the conditional random variable $R_{(t)}=X-t \mid X>t$, $t\ge 0$. Using \eqref{sur1}, the survival function of the residual lifetime $R_{(t)}$ for the NWTE distribution is given by
$$S_{R_{(t)}}(x)=\frac{S(x+t)}{S(t)}=\frac{1-\left[ 1-  \frac{\gamma}{1+\gamma} \mathrm{e}^{-\lambda (x+t)}\left\{1-\theta+\theta  \mathrm{e}^{-\lambda (x+t)}\right\}\right] ^{\frac{1}{\gamma}+1}}{1-\left[ 1-  \frac{\gamma}{1+\gamma} \mathrm{e}^{-\lambda t}(1-\theta+\theta  \mathrm{e}^{-\lambda  t})\right] ^{\frac{1}{\gamma}+1}}, \quad x\ge 0.$$
The associated cdf is given by
$$F_{R_{(t)}}(x)=\frac{\left[1-  \frac{\gamma}{1+\gamma} \mathrm{e}^{-\lambda (x+t)}\left\{1-\theta+\theta  \mathrm{e}^{-\lambda (x+t)}\right\}\right] ^{\frac{1}{\gamma}+1}-\left[ 1-  \frac{\gamma}{1+\gamma} \mathrm{e}^{-\lambda t}(1-\theta+\theta \mathrm{e}^{-\lambda t})\right] ^{\frac{1}{\gamma}+1}}{1-\left[1-  \frac{\gamma}{1+\gamma} \mathrm{e}^{-\lambda t}(1-\theta+\theta \mathrm{e}^{-\lambda t})\right] ^{\frac{1}{\gamma}+1}}.$$
The corresponding pdf is given by
\begin{eqnarray*}
f_{R_{(t)}}(x) = \frac{\lambda\left[ 1 -  \frac{\gamma}{1+\gamma} \mathrm{e}^{-\lambda (x+t)}\left\{1-\theta+\theta  \mathrm{e}^{-\lambda (x+t)}\right\}\right] ^{
\frac{1}{\gamma }}\mathrm{e}^{-\lambda  (x+t)} \left[ 1-\theta+2 \theta  \mathrm{e}^{-\lambda  (x+t)}\right]}{1-\left[1-  \frac{\gamma}{1+\gamma} \mathrm{e}^{-\lambda t}(1-\theta+\theta \mathrm{e}^{-\lambda t})\right]^{\frac{1}{\gamma}+1}}.
\end{eqnarray*}
The associated hrf is given by
\begin{eqnarray*}
h_{R_{(t)}}(x) = \frac{\lambda\left[1 -  \frac{\gamma}{1+\gamma} \mathrm{e}^{-\lambda (x+t)}\left\{1-\theta+\theta  \mathrm{e}^{-\lambda (x+t)}\right\}\right]^{\frac{1}{\gamma }}\mathrm{e}^{-\lambda  (x+t)} \left[ 1-\theta+2 \theta  \mathrm{e}^{-\lambda (x+t)}\right]}{1-\left[ 1-  \frac{\gamma}{1+\gamma} \mathrm{e}^{-\lambda (x+t)}\left\{1-\theta+\theta  \mathrm{e}^{-\lambda  (x+t)}\right\}\right]^{\frac{1}{\gamma}+1}}.
\end{eqnarray*}
The mean residual life is defined as
\begin{eqnarray*}
K(t)=E(R_{(t)})=\frac{1}{S(t)}\int_{t}^{+\infty}x f(x) \mathrm{d}x-t=\frac{1}{S(t)}\left[E(X)-J_1(t)\right]-t,
\end{eqnarray*}
where $f(x)$ is given by \eqref{pdf0}, $S(t)$ is mentioned in \eqref{surv0}, $E(X)$ is given by \eqref{mom} and $J_1(t)$ is stated in Lemma \ref{lem1}.

Further, the variance residual life is given by
\begin{align*}
V(t)&=Var(R_{(t)})=\frac{2}{S(t)}\int_{t}^{+\infty}xS(x)\mathrm{d}x-2tK(t)-\left[K(t)\right]^2\\
& = \frac{1}{S(t)}\left[E(X^2)-J_2(t)\right]-t^2-2 t K(t)-\left[K(t)\right]^2,
\end{align*}
where $E(X^2)$ is given by \eqref{mom} and $J_2(t)$ is given by Lemma \ref{lem1}. Some numerical values for the mean residual life are displayed in Table \ref{MRLF} for various choices of the parameters $\gamma $ and $\theta $ at the time points $t=1,3,5,7,10.$ It can be seen that, the mean residual life increases  with increasing the time points t, also decreases with increasing $\gamma$ and $\theta$.

\begin{table}
\centering
 \caption{Mean residual life function for arbitrary parameter values with $\lambda=1$.} \label{MRLF}
\scriptsize{\begin{tabular}{llllllllll}
\noalign{\smallskip}\hline\noalign{\smallskip}
$ \gamma \downarrow \theta=0.5$& $ t\rightarrow$  1.0 & 3.0 & 5.0 & 7.0 & 10 \\
\hline
0.1& 0.918095 & 0.982144 & 0.997423 & 0.999648 & 0.999982 &  \\
0.7& 0.900668 & 0.980089 & 0.997152 & 0.999612 & 0.999981 &  \\
1.1& 0.894316 & 0.979369 & 0.997057 & 0.999599 & 0.999980 &  \\
1.6& 0.889012 & 0.978778 & 0.996979 & 0.999588 & 0.999979 &  \\
2.0& 0.885996 & 0.978447 & 0.996936 & 0.999582 & 0.999979 &  \\ \hline
$ \theta  \downarrow \gamma=0.5 $& $ t\rightarrow$  1.0 & 3.0 & 5.0 & 7.0 & 10 \\
 \hline
-0.9 & 1.219460 & 1.027861 & 1.003735 & 1.000504 & 1.000025 \\
-0.5 & 1.162075 & 1.020915 & 1.002810 & 1.000380 & 1.000018 \\
-0.1 & 1.088250 & 1.011465 & 1.001542 & 1.000208 & 1.000010 \\
0.1  & 1.040826 & 1.004807 & 1.000638 & 1.000086 & 1.000004  \\
0.5  & 0.905030 & 0.980593 & 0.997218 & 0.999621 & 0.999981  \\
1.0  & 0.511566 & 0.500207 & 0.500004 & 0.500000 & 0.500001 \\ \hline
 \noalign{\smallskip}
\end{tabular}}
\end{table}

\subsection{Reversed residual life function}

The reverse residual life is described by the conditional random variable $\overline{R}_{(t)}=t-X \mid X\le t$, $t\ge 0$. Using \eqref{cdf0}, the survival function of the reversed residual lifetime $\overline{R}_{(t)}$ for the NWTE distribution is given by
$$S_{\overline{R}_{(t)}}(x)=\frac{F(t-x)}{F(t)}=\frac{\left[ 1-  \frac{\gamma}{1+\gamma} \mathrm{e}^{-\lambda (t-x)}\left\{1-\theta+\theta \mathrm{e}^{-\lambda (t-x)}\right\}\right] ^{\frac{1}{\gamma}+1}}{\left[1-  \frac{\gamma}{1+\gamma}  \mathrm{e}^{-\lambda t}(1-\theta+\theta \mathrm{e}^{-\lambda  t})\right] ^{\frac{1}{\gamma}+1}}, \quad 0\le x\le t.$$
The associated cdf is given by
$$F_{\overline{R}_{(t)}}(x)=\frac{\left[ 1-  \frac{\gamma}{1+\gamma} \mathrm{e}^{-\lambda t}(1-\theta+\theta \mathrm{e}^{-\lambda t})\right] ^{\frac{1}{\gamma}+1}-\left[ 1-  \frac{\gamma}{1+\gamma} \mathrm{e}^{-\lambda (t-x)}\left\{1-\theta+\theta  \mathrm{e}^{-\lambda  (t-x)}\right\}\right] ^{\frac{1}{\gamma}+1}}{\left[1-  \frac{\gamma}{1+\gamma} \mathrm{e}^{-\lambda t}(1-\theta+\theta \mathrm{e}^{-\lambda t})\right]^{\frac{1}{\gamma}+1}}.$$
The corresponding pdf is obtained as
\begin{eqnarray*}
f_{\overline{R}_{(t)}}(x) = \frac{\lambda\left[  1 -  \frac{\gamma}{1+\gamma}  \mathrm{e}^{-\lambda (t-x)}\left\{1-\theta+\theta  \mathrm{e}^{-\lambda   (t-x)}\right\}\right] ^{\frac{1}{\gamma }} \mathrm{e}^{-\lambda  (t-x)} \left[ 1-\theta+2 \theta  \mathrm{e}^{-\lambda  (t-x)}\right]}{\left[ 1-  \frac{\gamma}{1+\gamma}  \mathrm{e}^{-\lambda  t}(1-\theta+\theta  \mathrm{e}^{-\lambda  t})\right]^{\frac{1}{\gamma}+1}}.
\end{eqnarray*}
The associated hrf is given by
\begin{eqnarray*}
h_{\overline{R}_{(t)}}(x) = \frac{\lambda\left[ 1 -  \frac{\gamma}{1+\gamma} \mathrm{e}^{-\lambda (t-x)}\left\{1-\theta+\theta  \mathrm{e}^{-\lambda  (t-x)}\right\}\right] ^{
\frac{1}{\gamma }} \mathrm{e}^{-\lambda (t-x)} \left[ 1-\theta+2 \theta  \mathrm{e}^{-\lambda  (t-x)}\right]}{\left[ 1-  \frac{\gamma}{1+\gamma}  \mathrm{e}^{-\lambda  (t-x)}\left\{1-\theta+\theta  \mathrm{e}^{-\lambda  (t-x)}\right\}\right]^{\frac{1}{\gamma}+1}}.
\end{eqnarray*}
Moreover, the mean reversed residual life is defined as
\begin{eqnarray*}
L(t)=E(\overline{R}_{(t)})=t-\frac{1}{F(t)}\int_{0}^{t}xf(x)\mathrm{d}x=t-\frac{J_1(t)}{F(t)},
\end{eqnarray*}
where $f(x)$ is given by \eqref{pdf0}, $F(t)$ is defined by \eqref{cdf0} and $J_1(t)$ is given by Lemma \ref{lem1}.

Also, the variance reversed residual life is given by
\begin{align*}
W(t)&=Var(\overline{R}_{(t)})=2tL(t)-(L(t))^2-\frac{2}{F(t)}\int_{0}^{t}xF(x)\mathrm{d}x\\
& = 2tL(t)-(L(t))^2-t^2+\frac{J_2(t)}{F(t)},
\end{align*}
where $J_2(t)$ is given by Lemma \ref{lem1}.

In Table \ref{RMRLF}, we give some numerical values for the mean reversed residual life with different choices of the parameters $\gamma$ and $\theta$ at the time points $t=1,3,5,7,10.$ From this table, the mean reversed residual life increases with increasing the time points $t$ and with increasing $\gamma$ and $\theta$.

\begin{table}
\centering
 \caption{Mean reversed residual residue life function for arbitrary parameter values with $\lambda=1$.} \label{RMRLF}
\scriptsize{\begin{tabular}{llllllllll}
\noalign{\smallskip}\hline\noalign{\smallskip}
$ \gamma \downarrow \theta=0.5$& $ t\rightarrow$  1.0 & 3.0 & 5.0 & 7.0 & 10 \\
\hline
0.1& 0.568855 & 2.180127 & 4.080871 & 6.059352 & 9.054745  \\
0.7& 0.579174 & 2.217077 & 4.126584 & 6.107094 & 9.102934  \\
1.1& 0.583962 & 2.232563 & 4.145442 & 6.126726 & 9.122736  \\
1.6& 0.588572 & 2.246646 & 4.162445 & 6.144397 & 9.140554  \\
2.0& 0.591508 & 2.255231 & 4.172744 & 6.155087 & 9.151330  \\ \hline
$ \theta  \downarrow \gamma=0.5 $& $ t\rightarrow$  1.0 & 3.0 & 5.0 & 7.0 & 10 \\
 \hline
-0.9 & 0.415915 & 1.657881 & 3.395545 & 5.331680 & 8.317248 \\
-0.5 & 0.482478 & 1.812299 & 3.589379 & 5.536490 & 8.524715 \\
-0.1 & 0.528625 & 1.968720 & 3.792343 & 5.751628 & 8.742699 \\
0.1  & 0.546765 & 2.047690 & 3.897321 & 5.863170 & 8.855739  \\
0.5  & 0.576254 & 2.207155 & 4.114411 & 6.094402 & 9.090128  \\
1.0  & 0.604164 & 2.409363 & 4.399438 & 6.399130 & 9.399122 \\ \hline
 \noalign{\smallskip}
\end{tabular}}
\end{table}

\section{Estimation}\label{estim}
When the parameters $\lambda$, $\gamma$ and $\theta$ of the NWTE distribution need to be estimated, several estimation approaches are possible. In this section, we investigate the maximum likelihood estimates (MLEs) of  these parameters. Then we propose three goodness-of-fit statistics to compare the densities fitted to any data set.

\subsection{Maximum likelihood estimation}
Let $(x_1,\ldots,x_n)$ be a random samples of size $n$ from the NWTE distribution. Set $\Theta=(\lambda, \gamma, \theta)^{T}$, then the MLE of $\Theta$ can be determined by maximizing  the log-likelihood function $\ell(\Theta)$ given by
\begin{align*}
\ell(\Theta)&=n\left[\log(\lambda)+\log(\gamma +1)- \log\left([1+\gamma ]^{\frac{1}{\gamma }+1}-1\right)\right]+\frac{1}{\gamma }\sum _{i=1}^n \log\left[1+\gamma -\gamma  \mathrm{e}^{- \lambda  x_i}\left(1-\theta +\theta  \mathrm{e}^{- \lambda  x_i}\right)\right]\\ &\ \ \ -\lambda \sum _{i=1}^n  x_i+\sum _{i=1}^n \log\left(1-\theta +2\theta  \mathrm{e}^{- \lambda  x_i}\right).
\end{align*}
Alternatively, by differentiating $\ell(\Theta)$, the MLE of $\Theta$ can be obtained by solving  the nonlinear log-likelihood system equations  given by
\begin{align*}
\frac{\partial \ell (\Theta)}{\partial \lambda}& =\frac{n}{\lambda }-\sum _{i=1}^n x_i -\sum _{i=1}^n\frac{2 \mathrm{e}^{-\lambda  x_i} \theta  x_i}{1-\theta +2 \mathrm{e}^{-\lambda  x_i} \theta }+\sum _{i=1}^n \frac{2 \mathrm{e}^{-2 \lambda  x_i}  \theta +\mathrm{e}^{-\lambda  x_i} (1-\theta) }{1+\gamma -\mathrm{e}^{-\lambda  x_i} \gamma  \left(1-\theta +\mathrm{e}^{-\lambda  x_i} \theta \right)}x_i=0,\\
\frac{\partial \ell (\Theta )}{\partial \gamma }& =n \left[\frac{1}{1+\gamma }+\frac{(1+\gamma )^{1+\frac{1}{\gamma }} \left\{-\gamma +\log(1+\gamma)\right\}}{\gamma ^2 \left\{-1+(1+\gamma )^{1+\frac{1}{\gamma }}\right\}}\right]+\frac{1}{\gamma }\sum _{i=1}^n \frac{1-\mathrm{e}^{-\lambda  x_i} \left(1-\theta +\mathrm{e}^{-\lambda  x_i} \theta \right)}{1+\gamma -\mathrm{e}^{-\lambda  x_i} \gamma  \left(1-\theta +\mathrm{e}^{-\lambda  x_i} \theta \right)}\\ &\ \ \ -\frac{1}{\gamma ^2}\sum _{i=1}^n \log\left[1+\gamma -\mathrm{e}^{-\lambda  x_i} \gamma  \left(1-\theta +\mathrm{e}^{-\lambda  x_i} \theta \right)\right]=0, \\
\frac{\partial \ell (\Theta )}{\partial \theta}&=\sum _{i=1}^n \frac{-1+2  \mathrm{e}^{-\lambda  x_i}}{1-\theta +2  \mathrm{e}^{-\lambda  x_i} \theta }- \sum _{i=1}^n\frac{\mathrm{e}^{-\lambda  x_i} \left(-1+\mathrm{e}^{-\lambda  x_i}\right) }{1+\gamma -\mathrm{e}^{-\lambda  x_i} \gamma  \left(1-\theta +\mathrm{e}^{-\lambda  x_i} \theta \right)}=0.
\end{align*}
By solving the equations above simultaneously, we can obtain the MLE $\hat \Theta$ of $\Theta$, with components providing the MLEs $\hat \lambda$, $\hat \gamma$, $\hat \theta$ of $\lambda$, $\gamma$, $\theta$ respectively. Various numerical iterative techniques can be used for estimating these parameters. In this study, we consider the iterative algorithm inherent to the NMaximize command in the symbolic computational package  \textit{Mathematica}. 

Under some regularity conditions, the asymptotic normality of the MLEs is guaranteed; the asymptotic distribution of $(\hat{\Theta}-\Theta )$ is $N_{3}({0}_3,{I}(\Theta )^{-1})$, where ${I}(\Theta )=E({J}(\Theta ))$ denotes the expectation of the information matrix: ${J}(\Theta )=\{J_{rs}(\Theta )\}$, $(r,s)\in \{\lambda, \gamma, \theta\}$. Thus confidence intervals or Wald test can be constructed for the parameters.

Other estimation methods can be considered, as those performed in \cite{ali} for instance.   

\subsection{Goodness-of-fit statistics}
In order to evaluate the goodness-of-fit of the fitted models, we consider the Anderson-Darling statistics ($A^{\ast }$), the Cram\'{e}r-von Mises statistics ($W^{\ast }$) and the Kolmogrov-Smirnov statistics (K-S), given by
\begin{align*}
A^{\ast }& =\left( \frac{2.25}{n^{2}}+\frac{0.75}{n}+1\right) \left[ -n-
\frac{1}{n}\sum_{i=1}^{n}(2i-1) \log\left( z_{i}\left[
1-z_{n-i+1}\right] \right) \right] , \\
W^{\ast }& =\left( \frac{0.5}{n}+1\right) \left[ \sum_{i=1}^{n}\left(
z_{i}-\frac{2i-1}{2n}\right) ^{2}+\frac{1}{12n}\right], \\
K\text{-}S& =\max\left( \frac{i}{n}-z_{i},z_{i}-\frac{i-1}{n}\right),
\end{align*}
where $z_{i}=F(y_{i})$ and the $y_{i}^{  ,}s$ are the \emph{ordered} observations.
The associated $P$-values are determined. The better distribution in terms of fit is the one having the smallest statistics and largest $P$-values.

\section{Applications}\label{appli}
This section is devoted to the data analyses of two data sets in environmental sciences, namely hydrology, where we compare the fit of our new distributions and some well-known distributions.  The best model among them is then selected.

\subsection{Data fitting}\label{applications}
We consider the data sets: ``Ground-water data (GWD)'' described in Table 1 of Bhaumik and Gibbons \cite{Bhaumik and Gibbons} and ``Flood data (FD)'' described in Akinsete {\it et al.} \cite{Akinsete}. The data of GWD represent vinyl chloride concentrations ($n = 34$) collected from clean upgradient monitoring wells. The data of FD represent flood rates (for the years 1935--1973) ($n = 39$) for the Floyd River located in James, Iowa, USA. The descriptive statistics of both data sets are summarized in Table \ref{Descriptive}. From this table, the data are over-dispersed and having skewness and kurtosis. For each data set, the NWTE model is compared with the following distributions.
\begin{itemize}
\item The gamma distribution with pdf given by
\begin{eqnarray}
    f(x)=  \frac{x^{k-1}\mathrm{e}^{-\lambda ^{-1} x}}{ \lambda ^k \ \gamma(k)} , \quad  \ x>0, k>0, \lambda>0.\nonumber
    \end{eqnarray}
\item The Marshal-Olkin exponential distribution (MOE) \cite{Marshall and Olkin:1997} with a pdf given by
 \begin{eqnarray}
  f(x)=\frac{ \lambda   \beta   \mathrm{e}^{-\lambda  x}}{\left[1-(1-\beta ) \mathrm{e}^{-\lambda   x}\right]^2},\quad x>0 ,\beta>0 ,\lambda>0.\nonumber
    \end{eqnarray}
    \item The Nadarajah--Haghighi exponential distribution (NHE) \cite{Nadarajah and Haghighi} with a pdf given by
    \begin{eqnarray}
 f(x)=\alpha \lambda  (1+\lambda  x)^{\alpha -1}\mathrm{e}^{1-(1+\lambda  x)^{\alpha }}, \quad x>0 ,\lambda>0,   \alpha>0.\nonumber
    \end{eqnarray}
 \item The exponentiated exponential distribution (EE) \cite{gupta2} with a pdf given by
 \begin{align}
    f(x)=\alpha \lambda  \left(1-\mathrm{e}^{-\lambda  x}\right)^{\alpha -1}\mathrm{e}^{-\lambda  x}, \quad x>0,  \lambda>0,  \alpha >0.\nonumber
    \end{align}
 \item The transmuted Weibull distribution (TW)  \cite{Aryal and Tsokos:2011} with a pdf given by
 \begin{align}
    f(x)= \frac{\eta }{\sigma }\left(\frac{x}{\sigma }\right)^{\eta -1}\mathrm{e}^{-\left(\frac{x}{\sigma }\right)^{\eta }}\left[1-\lambda +2 \lambda   \mathrm{e}^{-\left(\frac{x}{\sigma }\right)^{\eta }}\right] , \quad  \ x>0, \eta>0,  \sigma>0, \ \lambda\in [-1,1].\nonumber
    \end{align}
 \item The transmuted generalized exponential distribution (TGE)  \cite{Khan_King_Hudson} with a pdf given by
 \begin{align}
    f(x)= \alpha   \theta   \mathrm{e}^{-\theta   x }\left(1-\mathrm{e}^{-\theta   x }\right)^{\alpha -1}\left[1+\lambda -2 \lambda   \left(1-\mathrm{e}^{-\theta   x}\right)^{\alpha }\right], \quad  \ x>0, \alpha>0,   \theta>0,  \lambda\in [-1,1].\nonumber
    \end{align}
\item The transmuted linear exponential distribution (TLE)  \cite{Ref tle} with a pdf given by
 \begin{align}
    f(x)= (\beta +\theta  x) \mathrm{e}^{-\left(\beta   x + \frac{\theta }{2} x^2\right)}\left[1-\lambda +2 \lambda   \mathrm{e}^{-\left(\beta  x + \frac{\theta }{2}x^2\right)}\right], \quad  \ x>0,  \beta>0,   \theta>0,   \lambda \in [-1,1]. \nonumber
    \end{align}
    \item The Kappa distribution  \cite{Mielke:1973} with a pdf given by
 \begin{align}
    f(x)= \frac{\alpha   \theta }{\beta }\left(\frac{x}{\beta }\right)^{\theta -1}\left[\alpha +\left(\frac{x}{\beta }\right)^{\alpha  \theta }\right]^{\frac{-(\alpha +1)}{\alpha }} , \quad  \ x>0, \alpha>0, \beta>0, \theta>0.\nonumber
    \end{align}
\end{itemize}

The MLEs with their standard errors  are given in Tables \ref{ground water contamination data} and \ref{flood data} for both data sets
along with the goodness-of-fit statistics for each distribution. We can see in Tables \ref{ground water contamination data} and \ref{flood data} that the NWTE distribution has the smallest statistics and the largest $P$-value; it provides the best fit among the considered distributions. This conclusion is confirmed again by Figure \ref{histogram}.

\begin{table*}[t]
\centering
 \caption{Descriptive statistics of both data sets.}\label{Descriptive}
\scriptsize{\begin{tabular}{lllllllllllll}
\noalign{\smallskip}\hline\noalign{\smallskip}
Data & Mean & Median  & SD & Kurtosis & Skewness & M1 & M2 & Min & Max \\  \hline
\textit{GWD} & 1.87941 & 1.15 & 1.95259 & 5.00541 & 1.60369 & 1.45692 & 0.8 & 0.1 & 8.0   \\ \hline
\textit{FD} & 6771.1 & 3570 & 11695.7 & 25.4436 & 4.55806 & 5872.77 & 2180 & 318 & 71500 \\
 \noalign{\smallskip}\hline
\end{tabular}}
\\ \text{\ \ \ \ \  \ \ \ \ \ \ \ SD = Standard Deviation ,  M1 = Mean deviation about the mean, }\\
\text{\ \  \ \ \ \ \ \ \  \ \ \ \ \ \ \ M2 = Mean deviation about the median}
\end{table*}

\begin{table}
\centering
 \caption{Comparison of fit of the NWTE distribution using different methods of estimation for GWD.}\label{ground water contamination data}
\scriptsize{\begin{tabular}{lllllllll}
\noalign{\smallskip}\hline\noalign{\smallskip}
&&& \small{MLE's} &&   \\
\hline
Distributions& Estimates& & & $A^* $&  $W^* $ & KS & $P$-value  \\
\hline
gamma($k, \lambda$)            & 1.062685   & 1.768549 && 0.320322 & 0.051617 & 0.097341 & 0.904069\\
                                & (0.228152) & (0.480351) \\
MOE($\beta, \lambda$)         & 0.822837   & 0.481811 && 0.246723 & 0.0322959 & 0.0876375 & 0.956463\\
                                & (0.484526) & (0.173277) \\
NHE($\alpha, \lambda$)        & 0.900308   & 0.631993 && 0.24749  & 0.0332962 & 0.0838071 & 0.97073  \\
                               & (0.344202)   & (0.415966)  \\
EE($\alpha, \lambda$)         & 1.076412   & 0.558049 && 0.325543 & 0.0528397 & 0.0977771 & 0.901191 \\
                               & (0.247363) & (0.124162) \\
TW($\eta, \sigma, \lambda$)    & 1.076390   & 2.392822 &  0.418645 & 0.255586  & 0.0384952 & 0.083499  & 0.971723 \\
                               & (0.146953) & (0.942422) & (0.606964) \\
TGE($\alpha, \theta, \lambda$) & 1.160258   & 0.480348 &  0.395341 & 0.261748  & 0.0407411 & 0.0889903 & 0.950556 \\
                               & (0.220835) & (0.217127) & (0.501833) \\
TLE($\beta, \theta, \lambda$)  & 0.404133   & 0.013541 &  0.391774 & 0.248782  & 0.0336508 & 0.0825401 & 0.97467 \\
                               & (0.267787) & (0.047316) & (0.373851) \\
Kappa($\alpha, \theta, \beta$) & 1.428222   & 1.236928 &  1.304859 & 0.248987  & 0.037697  & 0.0871235 & 0.958587\\
                               & (1.0934)   & (0.565213) & (0.489635) \\
NWTE($\lambda, \gamma,\theta$) & 0.465010   & 9.179478 &  0.344129 & \textbf{0.234947}&\textbf{0.0320984}&\textbf{0.0793788}&\textbf{0.982912}  \\
                               & (0.204453) &(50.0745) &  (0.75164)\\
 \noalign{\smallskip}\hline\noalign{\smallskip}
\end{tabular}}
\end{table}

\begin{table}
\centering
 \caption{Comparison of fit of the NWTE distribution using different methods of estimation for FD.}\label{flood data}
\scriptsize{\begin{tabular}{lllllllll}
\noalign{\smallskip}\hline\noalign{\smallskip}
&&& \small{MLE's} &&   \\
\hline
Distributions& Estimates& & & $A^* $&  $W^* $ & KS & $P$-value  \\
\hline
gamma($k, \lambda$)           & 0.919695   & 7362.32 && 1.2662 & 0.210505 & 0.147184 & 0.366821\\
                                & (0.182011) & (1906.34) \\
MOE($\beta, \lambda$) & 0.293231 & 0.000069 && 1.08796 & 0.148357 & 0.142366 & 0.407993\\
                          & (0.205071) & (0.000038) \\
NHE($\alpha, \lambda$) & 0.609712 & 0.000374 && 0.900554 & 0.117556 & 0.136633 & 0.460365 \\
                        & (0.127014)   & (0.000163)  \\
EE($\alpha, \lambda$) & 0.968901 & 0.000144 && 1.34484 & 0.23414 & 0.150306 & 0.341607\\
                        & (0.212611) & (0.000032) \\
TW($\eta, \sigma, \lambda$)  & 0.961730 & 10522.55 & 0.805980  & 0.780029 & 0.102975 & 0.11102 & 0.722328 \\
                            & (0.105471) & (2312.24) & (0.209698) \\
TGE($\alpha, \theta, \lambda$)  & 1.081744 & 0.000103 & 0.800145  & 0.795675 & 0.121992 & 0.108377 & 0.749403 \\
                                & (0.215061) & (0.000032) & (0.21265) \\
TLE($\beta, \theta, \lambda$)  & 0.000095 & 8.1$\times10^{-12}$ & 0.801661  & 0.782863 & 0.110629 & 0.108711 & 0.74601 \\
                               & (0.000021) & (1.3$\times10^{-9}$) & (0.207988) \\
Kappa($\alpha, \theta, \beta$)  &  0.038151 & 27.732540 & 1496.464 & 3.2369 & 0.658064 & 0.218663 & 0.048011 \\
                                & (0.095815)   & (67.9672) & (253.781) \\
NWTE($\lambda, \gamma,\theta$)  & 0.000097  & 29.109413 & 0.808284  &\textbf{0.775836}&\textbf{0.111574}&\textbf{0.108037}&\textbf{0.75285}  \\
                                &(0.000022) &(154.575) &  (0.199919)\\
 \noalign{\smallskip}\hline\noalign{\smallskip}
\end{tabular}}
\end{table}

\begin{figure}[h]
\centering
\includegraphics[width=5.5cm,height=5cm]{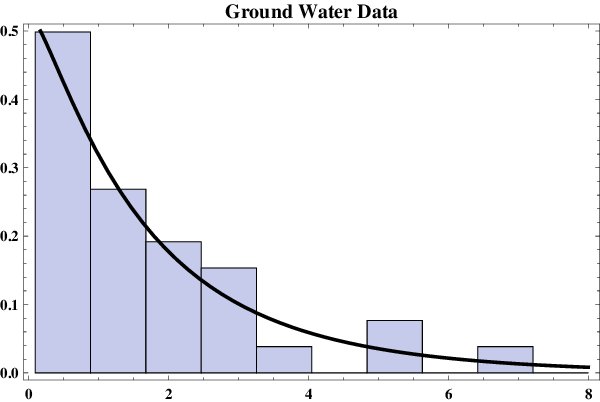}
\hspace{0.15cm}
\includegraphics[width=5.5cm,height=5cm]{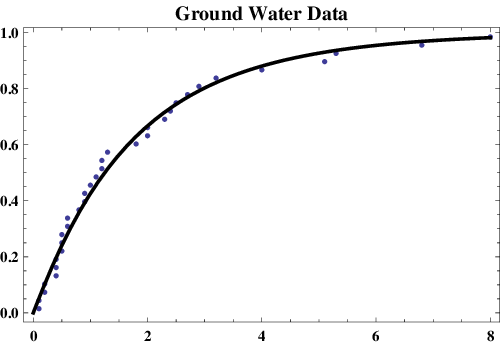}\\ \bigskip
\includegraphics[width=5.5cm,height=5cm]{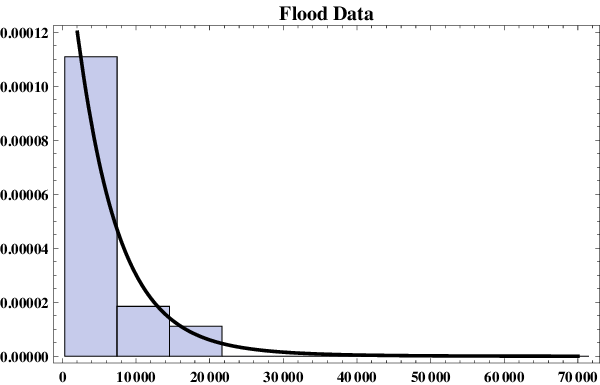}
\hspace{0.15cm}
\includegraphics[width=5.5cm,height=5cm]{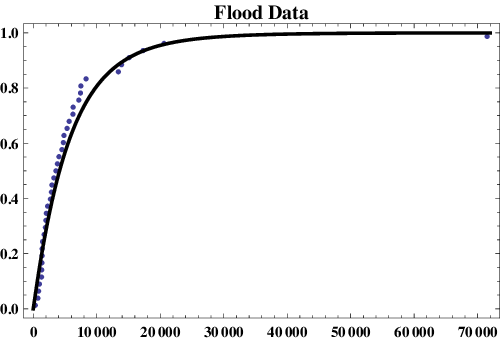}
\caption{Plots of the estimated pdfs and cdfs of the NWTE distribution, superimposed on the histograms and empirical cdfs, respectively, for the used data sets.}\label{histogram}
\end{figure}

\subsection{Hydrologic parameters}
The nice fit properties of the NWTE distribution motivates the determination of three important hydrologic parameters for the considered data sets: the return level, the conditional mean of the event data and the mean deviation about the return level. This recommends the NWTE as a hydrologic probability model, such as the most known distributions: Kappa and gamma distributions. 
\subsubsection{Return level}

A return period is an estimate of the likelihood of an event, such as a flood or a river discharge flow to occur. The probability, return period and return level of flood data and ground water contamination data can be estimated using the equation; $P(x_{T})=1-F(x_{T})$, $T=1/P(x_{T})$ and $x_{T}=F^{-1}\left( 1-\frac{1}{T}\right),$ respectively, where $F^{-1}\left( \cdot \right) $ is the inverse of
the cdf $F(x)$ and $P(x_{T})$ called exceedance probability (see, for instance, \cite{Ahammed, Nadarajah and Choi}). The return level $x_{T}$ under the NWTE distribution is obtained by
\begin{equation*}
x_{T}=- \frac{1}{\lambda}\log\left\lbrace 1-\left(\frac{1+\theta-\sqrt{(1+\theta)^2-4\frac{\theta}{\gamma}\left\lbrack\left\{ \left( 1-\frac{1}{T}\right)\left( \left[ 1+\gamma \right]^{\frac{1}{\gamma}+1}-1 \right) +1\right\}^{\frac{\gamma}{\gamma+1}}-1\right\rbrack }}{2   \theta}\right)\right\rbrace,
\end{equation*}
where $x_{T}>0$ and $T\geq 1.$ Table \ref{Return level estimates} provides estimates of the return level $x_{T}$ of the ground water contamination data and flood data, respectively, for the return periods $T=2,5,10,20,50,100,200$ years based on replacing the parameters $\lambda, \gamma, \theta$ by their ML estimates in Tables \ref{ground water contamination data} and \ref{flood data}.
Moreover, the return periods
for some largest values of the both data sets are reported in Table \ref{Return periods largest values} and computed using $T=1/P(x_{T})$, where $%
P(x_{T})=S(x_{T})$ is the estimated survival function of the NWTE distribution given by
\begin{eqnarray*}\label{survival function}
S(x_{T})=\frac{\left( 1+ \hat{\gamma} \right) ^{\frac{1}{\hat{\gamma} }+1} \left[ 1- \left\{ 1-  \frac{\hat{\gamma}}{1+\hat{\gamma}} \mathrm{e}^{-\hat{\lambda} x}\left(1-\hat{\theta}+\hat{\theta} \mathrm{e}^{-\hat{\lambda} x}\right)\right\} ^{\frac{1}{\hat{\gamma} }+1}\right]
}{\left( 1+ \hat{\gamma} \right) ^{\frac{1}{\hat{\gamma} }+1}-1},
\end{eqnarray*}
where $\hat{\lambda}, \hat{\gamma}, \hat{\theta},$ are the ML estimates corresponding the used data and are given in Tables \ref{ground water contamination data} and \ref{flood data}.
\subsubsection{Conditional mean of the event data}

The conditional mean of the event (GWD or FD) data based on equation \eqref{cm} is defined as
\begin{equation*}
E(X\mid X>Q)=\frac{1}{S(Q)}\int_{Q}^{\infty }x f(x) \mathrm{d}x,
\end{equation*}
where $S(x)$ is the survival function of the NWTE distribution and $Q$ is a value of the event. For example, for the GWD $E(X\mid X>8.0$ $m3/s)=10.1378$  and FD  $E(X\mid X>71500$  $mm)=81788.2.$

\subsubsection{Mean deviation about the return level}

The mean deviation about the return level is the mean of the distances of each value from their return level and it is a measure of the scatter in a population. The mean deviation about return level can be defined as
\begin{equation*}
\xi =\int_{0}^{\infty }\left\vert x -x_{T}\right\vert
f(x)\mathrm{d}x=2x_{T}F(x_{T})-x_{T}+\mu -2m(x_{T}),
\end{equation*}
where $m(x_{T})=\int_{0}^{x_{T}}x f(x)\mathrm{d}x$ and $f(x)$ is the pdf of the NWTE distribution. Table \ref{Return level estimates} provides mean deviation about the return level $m(x_{T})$ for the return periods $T=2,5,10,20,50,$ $100,200$ for the GWD and FD distributions, respectively, noting that we replace the parameters in $f(x)$ by their ML estimates for the corresponding data.

\begin{table}
\centering
 \caption{Return level estimates $\hat{x}_{T}$ for T and mean deviation about it.} \label{Return level estimates}
\scriptsize{\begin{tabular}{llllllll}
\noalign{\smallskip}\noalign{\smallskip}\hline
& GWD && FD \\ \hline
T & $x_{T} $ & \ \ \ $\xi $ & $x_{T}$ & \ \ \ $\xi $ \\  \hline
$2$ & \ \ \ 1.24872 & \ \ \ 1.31433 & 4143.93 &  4422.2\\
5 & \ \ \ 2.97016 & \ \ \ 1.9147 & 9802.81 & 6392.9 \\
10 & \ \ \ 4.34956 & \ \ \ 2.89705 & 14378.1 & 9653.1 \\
20 & \ \ \ 5.7765 & \ \ \ 4.11844 & 19306.6 & 13873.8 \\
50 &  \ \ \ 7.70555 & \ \ \ 5.92134 & 26491.6 & 20592.4 \\
100 & \ \ \ 9.18173 & \ \ \ 7.35495 & 32468.3 & 26397.4 \\
200 & \ \ \ 10.665 & \ \ \ 8.81682  & 38858.8 & 32695.9 \\
 \noalign{\smallskip}\hline
\end{tabular}}
\end{table}

\begin{table}
\centering
 \caption{Return periods for some largest values of the GWD and FD.} \label{Return periods largest values}
\scriptsize{\begin{tabular}{llllllll}
\noalign{\smallskip}\hline\noalign{\smallskip}
Values of the GWD & Return period& Values of the FD & Return period \\ \hline
4.0   & 8.41121     & 13900 & 9.32193 \\
5.1   & 14.4311     & 15100 & 11.1074 \\
5.3   & 15.8987     & 17300 & 15.1842 \\
6.8   & 32.5876     & 20600 & 23.7727 \\
8.0   & 57.4359     & 71500 & 5183.11 \\
\noalign{\smallskip}\hline
\end{tabular}}
\end{table}
\clearpage 

\section{Concluding remarks}\label{conclusion}
In this article, we introduce and study  a new three-parameter distribution, called the NWTE distribution, having the feature to combine the respective flexibility of the EGNB2 and transmuted exponential distributions.  Some of its mathematical properties are discussed, including the hazard rate function, moments, the moment generating function, the quantile function, various entropy measures and (reversed) residual life functions. Then, the NWTE is investigated from both the theoretical and practical aspects. In particular, the estimation of the parameters is performed with the method of maximum likelihood. By considering two environmental data sets, it is shown that it can provide better fits in comparison to eight well-established statistical models. Thanks to its high degree of flexibility, we believe that the NWTE model can found a place of choice for the analysis of data in other areas including engineering, medicine, science, ecology, biology and finance.


\end{document}